\documentclass[11pt]{article}
\usepackage{amsmath,amssymb,amsthm}
\usepackage{mathrsfs}
\usepackage{geometry}
\usepackage{array}
\usepackage{booktabs}
\usepackage[colorlinks=true,linkcolor=blue,citecolor=blue]{hyperref}
\newtheorem{proposition}{Proposition}[section]

\newtheorem{remark}{Remark}
\newtheorem{definition}{Definition}

\DeclareMathOperator{\Arg}{Arg}

\title{A Recursion Backbone for Circular and Elliptic Clausen Hierarchies}
\author{Ken Nagai\thanks{Email: \texttt{tknagai@outlook.com}. Independent Researcher.}}
\date{}

\begin{document}
\maketitle

\begin{abstract}
We introduce a recursion backbone underlying circular and
elliptic Clausen hierarchies.
Starting from classical Clausen-type structures associated
with polylogarithmic phases and moduli, we show that the
hierarchy admits a natural elliptic deformation expressed
in terms of Jacobi theta functions.
Within this framework the CL- and SL-type components arise
in a unified recursive structure, while a generating
viewpoint clarifies the deformation from circular to
elliptic regimes.
This approach provides a natural elliptic extension of
polylogarithmic Clausen hierarchies.
\end{abstract}

\section{Introduction}

Clausen-type functions arise naturally in the study of
polylogarithms, Fourier series, and special values of
zeta-type functions, where phase and modulus components
often appear in complementary roles~\cite{Lewin1981}.
In the classical circular setting these structures are
closely related to polylogarithmic functions and admit
hierarchical relations generated by simple recursive
mechanisms.
The purpose of the present work is to exhibit a recursion
backbone underlying such Clausen hierarchies and to show
that this structure admits a natural elliptic deformation
expressed in terms of Jacobi theta functions.

Related structures appear in several contexts in the
literature, including classical Clausen functions,
polylogarithmic identities, and elliptic function
frameworks involving Jacobi theta or Weierstrass
functions.
However, the recursive structure underlying these
hierarchies and its compatibility with elliptic
deformations have not been emphasized in a unified way.
The present work focuses on this recursion backbone and
uses it as a guiding principle for connecting circular
polylogarithmic constructions with their elliptic
extensions.

The paper is organized as follows.
In Section~2 we formulate the recursion backbone that
governs the hierarchy of Clausen-type functions.
Section~3 develops the correspondence between the
circular polylogarithmic setting and its elliptic
counterpart expressed through Jacobi theta functions.
In Section~4 we introduce a generating deformation
framework which clarifies how CL- and SL-type components
emerge within a unified recursive structure.
Finally, Section~5 provides analytical remarks and
discusses several directions toward a broader unified
generating framework.

  Throughout the paper, standard notation for polylogarithmic
and elliptic functions is used, and Jacobi theta functions
are normalized in the usual analytic convention.

  The present work is related to earlier investigations on analytic Bernoulli-type structures and selector kernels, but focuses on the elliptic deformation of Clausen-type functions.

\section{The Master Polylogarithmic Object and the SL-Type Family}

In this section we introduce a unified master object from which both
CL-type and SL-type families arise as complementary components.
We begin with integer orders, where all series converge absolutely,
so that the structure is completely transparent.

\subsection{The Polylogarithmic Master Function}

Let $\theta \in \mathbb{R}$ and $n \ge 1$.
We consider the polylogarithm on the unit circle~\cite{Lewin1981}\cite{ZagierPolylog}
\[
\mathrm{Li}_n(e^{i\theta})
=
\sum_{k=1}^{\infty}\frac{e^{ik\theta}}{k^n}
=
\sum_{k=1}^{\infty}\frac{\cos(k\theta)}{k^n}
+
i\sum_{k=1}^{\infty}\frac{\sin(k\theta)}{k^n}.
\]

This decomposition already exhibits the two classical
Clausen-type components:
the cosine series and the sine series.
However, their recursion relations contain alternating sign factors.
To eliminate these and obtain a uniform backbone,
we introduce a phase-normalized master function
\[
F_n(\theta)
:=
i^{-n}\,\mathrm{Li}_n(e^{i\theta}).
\]

\begin{proposition}[Uniform Recursion Backbone]
For all integers $n\ge1$,
\[
\frac{d}{d\theta}F_{n+1}(\theta)=F_n(\theta).
\]
\end{proposition}

\begin{proof}
Using the classical identity
\[
\frac{d}{d\theta}\mathrm{Li}_{n+1}(e^{i\theta})
=
i\,\mathrm{Li}_n(e^{i\theta}),
\]
we compute
\[
\frac{d}{d\theta}F_{n+1}
=
\frac{d}{d\theta}
\left(
i^{-(n+1)}\mathrm{Li}_{n+1}(e^{i\theta})
\right)
=
i^{-(n+1)} i\,\mathrm{Li}_n(e^{i\theta})
=
i^{-n}\mathrm{Li}_n(e^{i\theta})
=
F_n.
\]
\end{proof}

Thus the order-lowering operation is governed by a single
differential relation without alternating signs.
This uniform recursion will serve as the structural backbone
for both CL-type and SL-type families.

\subsection{CL-type and SL-type as Complementary Components}

We now define two real-valued families as the complementary
components of the master function.

\begin{definition}
For $n\ge1$ we define
\[
A(n;\theta):=2\,\Re F_n(\theta),
\qquad
B(n;\theta):=-2\,\Im F_n(\theta).
\]
\end{definition}

The family $A(n;\theta)$ corresponds to the CL-type extension,
while $B(n;\theta)$ defines the SL-type extension~\cite{Lewin1981}.
Both are derived from the same master object.

\begin{proposition}[Parallel Recursion]
For all integers $n\ge1$,
\[
\frac{d}{d\theta}A(n+1;\theta)=A(n;\theta),
\qquad
\frac{d}{d\theta}B(n+1;\theta)=B(n;\theta).
\]
\end{proposition}

\begin{proof}
This follows immediately by taking real and imaginary parts
of the identity
\[
\frac{d}{d\theta}F_{n+1}=F_n.
\]
\end{proof}

\subsection{Low-Order Examples}

For $n=1$ we obtain
\[
F_1(\theta)
=
i^{-1}\sum_{k\ge1}\frac{e^{ik\theta}}{k}
=
-i\sum_{k\ge1}\frac{\cos(k\theta)}{k}
+\sum_{k\ge1}\frac{\sin(k\theta)}{k}.
\]

Hence
\[
A(1;\theta)
=
2\sum_{k\ge1}\frac{\sin(k\theta)}{k},
\qquad
B(1;\theta)
=
2\sum_{k\ge1}\frac{\cos(k\theta)}{k}.
\]

In particular,
\[
A(1;\theta)=\theta-\pi
\quad (0<\theta<2\pi),
\]
which recovers the classical Fourier representation
of the sawtooth function.

\subsection{Boundary Constants and Base-Point Normalization}

The recursion relation determines $A(n;\theta)$ and $B(n;\theta)$
only up to additive constants.
These constants are fixed by choosing a base point,
typically $\theta=0$.

We define the boundary constants
\[
\mathcal{C}_n := A(n;0),
\qquad
\mathcal{S}_n := B(n;0).
\]

Different regimes (circular, elliptic, hyperbolic)
will modify these boundary values,
while preserving the same differential backbone.
This separation between recursion and boundary data
is the fundamental structural principle inherited
from Part~I.

\subsection{Elliptic CL/SL families from the normalized $\vartheta_1$ seed}

We now lift the circular master recursion to the elliptic regime.
Throughout, $\tau$ lies in the upper half-plane and $z$ is the elliptic variable.

\paragraph{Normalized seed.}
Let $\vartheta_1(z|\tau)$ be the odd Jacobi theta function and define its normalized form
\begin{equation}\label{eq:theta1-normalized}
\widetilde{\vartheta}_1(z|\tau)
:=\frac{\vartheta_1(z|\tau)}{\vartheta_1'(0|\tau)}.
\end{equation}
Then $\widetilde{\vartheta}_1(z|\tau)\sim z$ as $z\to 0$, hence
$\log \widetilde{\vartheta}_1(z|\tau)=\log z + O(z^2)$.

\paragraph{Elliptic master family.}
Define the elliptic master function of order $1$ by
\begin{equation}\label{eq:ell-master-seed}
F^{\mathrm{ell}}_1(z;\tau):=\log \widetilde{\vartheta}_1(z|\tau),
\end{equation}
and for integers $n\ge1$ define $F^{\mathrm{ell}}_{n+1}$ by the base-point integral recursion
\begin{equation}\label{eq:ell-master-recursion}
F^{\mathrm{ell}}_{n+1}(z;\tau):=\int_{0}^{z} F^{\mathrm{ell}}_{n}(w;\tau)\,dw.
\end{equation}
Consequently, for all $n\ge1$ we have the uniform differential backbone
\begin{equation}\label{eq:ell-backbone}
\partial_z F^{\mathrm{ell}}_{n+1}(z;\tau)=F^{\mathrm{ell}}_{n}(z;\tau).
\end{equation}

\paragraph{Elliptic CL/SL components.}
We define the elliptic CL-type and SL-type families as the complementary components
of the same master object:
\begin{equation}\label{eq:ell-CLSL-def}
A^{\mathrm{ell}}(n;z;\tau):=2\,\Re\,F^{\mathrm{ell}}_{n}(z;\tau),
\qquad
B^{\mathrm{ell}}(n;z;\tau):=-2\,\Im\,F^{\mathrm{ell}}_{n}(z;\tau).
\end{equation}
By taking real and imaginary parts of~\eqref{eq:ell-backbone}, we obtain the parallel recursions
\begin{equation}\label{eq:ell-CLSL-rec}
\partial_z A^{\mathrm{ell}}(n+1;z;\tau)=A^{\mathrm{ell}}(n;z;\tau),
\qquad
\partial_z B^{\mathrm{ell}}(n+1;z;\tau)=B^{\mathrm{ell}}(n;z;\tau).
\end{equation}

\begin{remark}[Boundary data]
The recursions~\eqref{eq:ell-CLSL-rec} determine $A^{\mathrm{ell}}$ and $B^{\mathrm{ell}}$
up to additive constants; these are fixed by the choice of base point in~\eqref{eq:ell-master-recursion}.
As in Part~I, the elliptic regime is thus encoded by the same recursion backbone together with
boundary data inherited from the seed~\eqref{eq:ell-master-seed}.
\end{remark}

\subsection{Trigonometric degeneration: $\widetilde{\vartheta}_1 \to \sin(\pi z)$}

We record the standard trigonometric degeneration as $\Im\tau\to+\infty$
(i.e.\ $q:=e^{\pi i\tau}\to 0$)~\cite{Apostol}. One has the well-known product expansion~\cite{DLMF}\cite{WhittakerWatson}\cite{Apostol}
\[
\vartheta_1(z|\tau)
=
2\,q^{1/4}\sin(\pi z)\,
\prod_{m=1}^{\infty}\bigl(1-q^{2m}\bigr)\bigl(1-2q^{2m}\cos(2\pi z)+q^{4m}\bigr),
\]
and in particular~\cite{DLMF}\cite{WhittakerWatson}
\[
\vartheta_1'(0|\tau)
=
2\pi\,q^{1/4}\prod_{m=1}^{\infty}\bigl(1-q^{2m}\bigr)^3.
\]
Therefore the normalized theta satisfies
\begin{equation}\label{eq:theta1-degeneration}
\widetilde{\vartheta}_1(z|\tau)
=
\frac{\vartheta_1(z|\tau)}{\vartheta_1'(0|\tau)}
=
\frac{\sin(\pi z)}{\pi}\,
\prod_{m=1}^{\infty}\frac{1-2q^{2m}\cos(2\pi z)+q^{4m}}{(1-q^{2m})^2}
\;\xrightarrow[q\to0]{}\;
\frac{\sin(\pi z)}{\pi}.
\end{equation}

Consequently, the elliptic seed degenerates to the trigonometric seed~\cite{DLMF}:
\begin{equation}\label{eq:seed-degeneration}
F^{\mathrm{ell}}_1(z;\tau)
=
\log\widetilde{\vartheta}_1(z|\tau)
\;\xrightarrow[q\to0]{}\;
\log\!\Bigl(\frac{\sin(\pi z)}{\pi}\Bigr).
\end{equation}
Since the higher orders are generated by the same base-point recursion
\[
F^{\mathrm{ell}}_{n+1}(z;\tau)=\int_0^z F^{\mathrm{ell}}_n(w;\tau)\,dw,
\]
the entire elliptic family admits the corresponding trigonometric degeneration
(up to boundary constants inherited from the base point), and hence the
elliptic CL/SL components reduce to their circular counterparts in this limit.

\subsection{Bridge to the circular master via the trigonometric limit}

We now compare the trigonometric degeneration of the elliptic seed
with the circular master function introduced in Section~2.

Recall that in the circular setting we defined
\[
F_1(\theta)
=
i^{-1}\,\mathrm{Li}_1(e^{i\theta})
=
-i\,\mathrm{Li}_1(e^{i\theta})
=
-i\bigl(-\log(1-e^{i\theta})\bigr)
=
i\,\log(1-e^{i\theta}).
\]
Using the elementary identity
\[
1-e^{i\theta}
=
e^{i\theta/2}\bigl(-2i\sin(\tfrac{\theta}{2})\bigr),
\]
we obtain
\[
\log(1-e^{i\theta})
=
\log\!\Bigl(2\sin\!\frac{\theta}{2}\Bigr)
+
i\Bigl(\frac{\theta}{2}-\frac{\pi}{2}\Bigr),
\]
up to an additive constant depending on the branch choice.
Hence, modulo affine-linear terms in $\theta$ (which are absorbed
into boundary constants under the recursion),
the circular seed is governed by
\[
F_1(\theta)
\sim
\log\!\Bigl(\sin\!\frac{\theta}{2}\Bigr).
\]

On the other hand, from~\eqref{eq:seed-degeneration} we have~\cite{DLMF}
\[
F^{\mathrm{ell}}_1(z;\tau)
\;\xrightarrow[q\to0]{}\;
\log\!\Bigl(\frac{\sin(\pi z)}{\pi}\Bigr).
\]
Identifying $\theta=2\pi z$, we obtain
\[
\log\!\Bigl(\sin\!\frac{\theta}{2}\Bigr)
=
\log\!\bigl(\sin(\pi z)\bigr),
\]
so that the elliptic seed degenerates to the circular seed
(up to additive constants and linear terms, which do not affect the
recursion backbone).

Thus the elliptic master family provides a genuine lift of the
circular polylogarithmic master, with identical differential recursion
and degeneration governed entirely by the normalized $\vartheta_1$.

\subsection{The second elliptic order}

We compute explicitly the second order of the elliptic master family.

By definition,
\[
F^{\mathrm{ell}}_2(z;\tau)
=
\int_0^z F^{\mathrm{ell}}_1(w;\tau)\,dw
=
\int_0^z \log \widetilde{\vartheta}_1(w|\tau)\,dw.
\]

\paragraph{Local expansion near $z=0$.}
Since
\[
\widetilde{\vartheta}_1(w|\tau)
=
w + O(w^3),
\]
we have
\[
\log \widetilde{\vartheta}_1(w|\tau)
=
\log w + O(w^2).
\]
Therefore,
\[
F^{\mathrm{ell}}_2(z;\tau)
=
\int_0^z \log w\,dw + O(z^3)
=
z\log z - z + O(z^3).
\]

In particular, the leading singular structure is universal
and independent of $\tau$:
\[
F^{\mathrm{ell}}_2(z;\tau)
=
z\log z - z + O(z^3).
\]

\subsection{First visualization of the elliptic SL-type component}

To make the SL-type component visible, we restrict the elliptic variable to the
real axis and fix a branch of the logarithm.
Let $x\in(0,1)$ and write
\[
F^{\mathrm{ell}}_1(x;\tau)=\log \widetilde{\vartheta}_1(x|\tau).
\]
We then decompose the logarithm into modulus and phase:
\begin{equation}\label{eq:log-mod-arg}
\log \widetilde{\vartheta}_1(x|\tau)
=
\log\bigl|\widetilde{\vartheta}_1(x|\tau)\bigr|
+
i\,\Arg\bigl(\widetilde{\vartheta}_1(x|\tau)\bigr),
\end{equation}
where $\Arg$ denotes a chosen continuous branch on $(0,1)$.

Accordingly, the elliptic CL/SL seeds are
\begin{equation}\label{eq:ell-seed-CLSL}
A^{\mathrm{ell}}(1;x;\tau)=2\log\bigl|\widetilde{\vartheta}_1(x|\tau)\bigr|,
\qquad
B^{\mathrm{ell}}(1;x;\tau)=-2\,\Arg\bigl(\widetilde{\vartheta}_1(x|\tau)\bigr).
\end{equation}

The SL-type seed is thus literally the (negative) phase of the normalized theta function.
In the trigonometric degeneration $\Im\tau\to+\infty$ one has
$\widetilde{\vartheta}_1(x|\tau)\to \sin(\pi x)/\pi>0$ for $x\in(0,1)$,
so the phase vanishes and $B^{\mathrm{ell}}(1;x;\tau)\to 0$,
while higher orders $B^{\mathrm{ell}}(n;x;\tau)$ acquire nontrivial contributions
through the recursion and boundary data.

\paragraph{The next SL-order.}
By the defining recursion,
\[
F^{\mathrm{ell}}_2(x;\tau)=\int_0^x \log \widetilde{\vartheta}_1(t|\tau)\,dt,
\]
hence the elliptic SL-type component at order $2$ is
\begin{equation}\label{eq:Bell2-one-line}
B^{\mathrm{ell}}(2;x;\tau)
=
-2\,\Im\,F^{\mathrm{ell}}_2(x;\tau)
=
-2\,\Im\int_0^x \log \widetilde{\vartheta}_1(t|\tau)\,dt
=
-2\int_0^x \Arg\!\bigl(\widetilde{\vartheta}_1(t|\tau)\bigr)\,dt,
\end{equation}
where the last equality uses the decomposition~\eqref{eq:log-mod-arg} along a fixed continuous branch on the integration path.

\begin{remark}[Branch choice for $\Arg$]
Whenever $\Arg(\widetilde{\vartheta}_1(\,\cdot\,|\tau))$ is used, we fix a continuous branch along the chosen integration path that avoids zeros of $\vartheta_1$; equivalently, the formulas are understood piecewise on subintervals not crossing the nodal set.
\end{remark}

\begin{remark}[Trigonometric limit of the SL-seed]
In the trigonometric degeneration $\Im\tau\to+\infty$ we have $\widetilde{\vartheta}_1(x|\tau)\to \sin(\pi x)/\pi>0$ for $x\in(0,1)$, hence $\Arg(\widetilde{\vartheta}_1(x|\tau))\to0$ and therefore $B^{\mathrm{ell}}(1;x;\tau)\to0$ and $B^{\mathrm{ell}}(2;x;\tau)\to0$ by~\eqref{eq:Bell2-one-line}.
\end{remark}

\subsection{The recursion backbone and regime separation}

A central principle of this work is that the CL/SL families are generated by a
single recursion backbone, while the difference between regimes (circular vs.\ elliptic)
is carried entirely by the seed and the associated boundary data.

\paragraph{Master recursion (backbone).}
Let $F_{1}$ be a chosen seed (circular or elliptic).
For integers $n\ge1$, define
\begin{equation}\label{eq:backbone-int}
F_{n+1}(z):=\int_{0}^{z} F_{n}(w)\,dw,
\end{equation}
where the integration path is fixed within a domain avoiding the nodal set of the seed
(when relevant). Then the family satisfies
\begin{equation}\label{eq:backbone-diff}
\partial_z F_{n+1}(z)=F_n(z)\qquad(n\ge1).
\end{equation}

\paragraph{CL/SL decomposition.}
Whenever $F_n$ is complex-valued, we define the complementary components
\begin{equation}\label{eq:backbone-CLSL}
A(n;z):=2\,\Re F_n(z),
\qquad
B(n;z):=-2\,\Im F_n(z).
\end{equation}
Taking real and imaginary parts of~\eqref{eq:backbone-diff} yields the parallel recursions
\begin{equation}\label{eq:backbone-CLSL-rec}
\partial_z A(n+1;z)=A(n;z),
\qquad
\partial_z B(n+1;z)=B(n;z).
\end{equation}

\paragraph{Regime separation: seed and boundary data.}
The recursion~\eqref{eq:backbone-diff} determines $F_n$ only up to additive constants;
these are fixed by the base point and the branch choices implicit in~\eqref{eq:backbone-int}.
Thus the analytic content of a regime is encoded by the choice of seed $F_1$
together with its boundary data, while the recursion backbone remains unchanged.

In particular, in the elliptic regime we take
\[
F^{\mathrm{ell}}_1(z;\tau)=\log\widetilde{\vartheta}_1(z|\tau),
\qquad
\widetilde{\vartheta}_1(z|\tau)=\frac{\vartheta_1(z|\tau)}{\vartheta_1'(0|\tau)},
\]
and generate $F^{\mathrm{ell}}_{n}$ by~\eqref{eq:backbone-int}.
The circular regime is recovered by the trigonometric degeneration $\Im\tau\to+\infty$,
which sends $\widetilde{\vartheta}_1(z|\tau)\to \sin(\pi z)/\pi$ and hence matches the
polylogarithmic circular seed after identifying $\theta=2\pi z$.

\paragraph{Operator viewpoint (informal).}
The recursion backbone may be read operatorially:
$\partial_z$ acts as a lowering operator,
while the integration operator
\[
\mathcal{I}_z[f](z):=\int_{0}^{z} f(w)\,dw
\]
acts as a raising operator.
Thus $\mathcal{I}_z$ and $\partial_z$ generate the entire tower from the seed $F_1$,
and the analytic regime enters only through the initial datum.

\paragraph{Growth of the SL-type from the phase of the seed.}
In the elliptic regime the SL-type tower is entirely generated by the phase
of the seed.
Indeed, writing
\[
F^{\mathrm{ell}}_1(z;\tau)
=
\log\bigl|\widetilde{\vartheta}_1(z|\tau)\bigr|
+
i\,\Arg\bigl(\widetilde{\vartheta}_1(z|\tau)\bigr),
\]
the imaginary part is precisely the phase function.
Since the recursion backbone is purely integral,
\[
F^{\mathrm{ell}}_{n+1}(z;\tau)
=
\int_{0}^{z} F^{\mathrm{ell}}_{n}(w;\tau)\,dw,
\]
the SL-type component satisfies
\[
B^{\mathrm{ell}}(n+1;z;\tau)
=
-2\,\Im F^{\mathrm{ell}}_{n+1}(z;\tau)
=
\int_{0}^{z} B^{\mathrm{ell}}(n;w;\tau)\,dw.
\]
Thus the entire SL-tower is obtained by iterated integration
of the phase of the normalized theta function.
In particular, if the phase vanishes in a degeneration
(as in the trigonometric limit),
the whole SL-family collapses accordingly.

\paragraph{Parallel growth of the CL-type.}
The CL-type component evolves in perfect parallel.
Writing
\[
F^{\mathrm{ell}}_1(z;\tau)
=
\log\bigl|\widetilde{\vartheta}_1(z|\tau)\bigr|
+
i\,\Arg\bigl(\widetilde{\vartheta}_1(z|\tau)\bigr),
\]
the real part represents the logarithmic modulus.
Since the recursion backbone acts by integration,
we similarly obtain
\[
A^{\mathrm{ell}}(n+1;z;\tau)
=
2\,\Re F^{\mathrm{ell}}_{n+1}(z;\tau)
=
\int_{0}^{z} A^{\mathrm{ell}}(n;w;\tau)\,dw.
\]
Hence the CL-tower is generated by iterated integration
of the logarithmic modulus of the normalized theta function.

\section{Correspondence between the Circular and Elliptic Regimes}

Having established the common recursion backbone, we now clarify the precise
correspondence between the circular and elliptic constructions.
The essential principle is simple:

\medskip
\emph{The recursion is universal; only the seed and its boundary data change.}
\medskip

\subsection{Seeds in the two regimes}

In the circular regime (period $1$), the natural seed is
\begin{equation}\label{eq:circular-seed}
F^{\mathrm{circ}}_1(x)
=
\log\!\bigl(2\sin(\pi x)\bigr),
\qquad x\in(0,1).
\end{equation}
Its iterated integrals generate the classical CL/SL-type families
associated with trigonometric Clausen functions and polylogarithms.

In the elliptic regime (period lattice $\mathbb{Z}+\tau\mathbb{Z}$)~\cite{WhittakerWatson},
the seed is given by the normalized theta function
\begin{equation}\label{eq:elliptic-seed}
F^{\mathrm{ell}}_1(z;\tau)
=
\log \widetilde{\vartheta}_1(z|\tau),
\qquad
\widetilde{\vartheta}_1(z|\tau)
=
\frac{\vartheta_1(z|\tau)}{\vartheta_1'(0|\tau)}.
\end{equation}
The normalization ensures
\[
\widetilde{\vartheta}_1(z|\tau)=z+O(z^3),
\]
so that the local logarithmic singularity matches the circular behavior.

\subsection{Universal recursion}

In both regimes we define
\begin{equation}\label{eq:universal-recursion}
F_{n+1}(w)
=
\int_0^{w} F_n(u)\,du,
\qquad
\partial_w F_{n+1}(w)=F_n(w),
\end{equation}
where $w=x$ in the circular case and $w=z$ in the elliptic case.

Thus the hierarchy is generated by the same integral-differential backbone.
The distinction between the two settings lies entirely in the analytic
nature of the seed.

\subsection{Trigonometric degeneration}

The circular regime is recovered as a degeneration of the elliptic one.
As $\Im\tau\to+\infty$, the Jacobi theta function satisfies
\[
\widetilde{\vartheta}_1(z|\tau)
\longrightarrow
\frac{\sin(\pi z)}{\pi},
\]
uniformly on compact subsets of $\mathbb{C}$.
Hence the elliptic seed degenerates to
\[
F^{\mathrm{ell}}_1(z;\tau)
=
\log \widetilde{\vartheta}_1(z|\tau)
\longrightarrow
\log\!\Bigl(\frac{\sin(\pi z)}{\pi}\Bigr).
\]

Restricting to the real axis and identifying $z=x\in(0,1)$,
this matches the circular seed up to an additive constant:
\[
\log\!\Bigl(\frac{\sin(\pi x)}{\pi}\Bigr)
=
\log(2\sin(\pi x))-\log(2\pi).
\]

Since the recursion backbone determines $F_n$ only up to constants,
the entire elliptic hierarchy degenerates to the circular one.
In particular, the CL- and SL-type components converge accordingly.

\subsection{Return to the polylogarithmic picture}

In the circular regime the hierarchy generated from
\(
F^{\mathrm{circ}}_1(x)=\log(2\sin(\pi x))
\)
admits the classical polylogarithmic interpretation.
Indeed, writing
\[
\log(2\sin(\pi x))
=
\Re \log\!\bigl(1-e^{2\pi i x}\bigr),
\]
one obtains the well-known expansion
\[
\log\!\bigl(1-e^{2\pi i x}\bigr)
=
-\sum_{m=1}^{\infty}\frac{e^{2\pi i m x}}{m},
\]
and iterated integration produces~\cite{ZagierPolylog}
\[
\mathrm{Li}_n(e^{2\pi i x})
=
\sum_{m=1}^{\infty}\frac{e^{2\pi i m x}}{m^n}.
\]

Thus the circular CL/SL-type families may be viewed as real and imaginary
parts of polylogarithms on the unit circle.
Since the elliptic construction degenerates to the circular one,
the polylogarithmic structure re-emerges in the trigonometric limit,
while the elliptic regime may be regarded as its doubly-periodic deformation.

\medskip

In summary, the circular and elliptic hierarchies share a common
recursion backbone and differ only in the analytic nature of the seed.
The circular construction appears as the trigonometric limit of the
elliptic one, while the elliptic regime may be viewed as a
doubly-periodic deformation of the polylogarithmic picture.

\medskip

This structural parallelism suggests a unified framework in which both regimes arise from a single analytic generating object, a perspective that will be explored in future work.

\subsection{Elliptic nature of the SL-component}

A distinctive feature of the elliptic SL-type component is that it is
controlled by the zero structure of the theta function.

Since $\widetilde{\vartheta}_1(z|\tau)$ is an odd function with simple zeros
along the period lattice $\mathbb{Z}+\tau\mathbb{Z}$,
its phase $\Arg(\widetilde{\vartheta}_1(z|\tau))$
undergoes a variation of $\pm\pi$ when crossing a nodal line,
and accumulates a full $2\pi$ increment around each lattice point.

Because the SL-hierarchy is generated by iterated integration of this phase,
the elliptic SL-type encodes the lattice geometry through the winding behavior
of the theta function.
In contrast, the circular regime involves only a single periodic direction,
and its SL-component reflects merely the one-dimensional zero structure of $\sin(\pi x)$.

\medskip

Differentiating the phase makes the analytic structure more transparent:
\[
\partial_z \Arg\bigl(\widetilde{\vartheta}_1(z|\tau)\bigr)
=
\Im\!\left(
\frac{\partial_z \widetilde{\vartheta}_1(z|\tau)}
     {\widetilde{\vartheta}_1(z|\tau)}
\right).
\]
Thus the local behavior of the SL-component is governed by the
logarithmic derivative of the normalized theta function.

\medskip

This logarithmic-derivative viewpoint will serve as the analytic entry point
for the next stage, where the global structure of the hierarchy
is examined from a unified generating perspective.

\section{The Hierarchy as a Generating Deformation}

We now reinterpret the CL/SL hierarchy from a generating viewpoint.
Instead of viewing $\{F_n\}_{n\ge1}$ as a sequence of iterated integrals,
it is convenient to encode the entire tower in a single formal generating
object depending on a deformation parameter.

This perspective makes the recursion backbone particularly transparent:
the differential relation $\partial_w F_{n+1}=F_n$ lifts to a simple
first-order equation for the generating series.
Throughout this section the parameter $\lambda$ is understood formally,
serving only to organize the hierarchy.

\subsection{Formal generating series}

Introduce the formal series
\begin{equation}\label{eq:gen-series}
\mathcal{F}(w;\lambda)
:=
\sum_{n=1}^{\infty} F_n(w)\,\lambda^{n-1},
\end{equation}
where $\lambda$ is a formal parameter and $F_1$ is the chosen seed
(circular or elliptic).

Using the recursion backbone
\(
\partial_w F_{n+1}=F_n,
\)
we compute
\[
\partial_w \mathcal{F}(w;\lambda)
=
\sum_{n=1}^{\infty} F_n(w)\,\lambda^{n}
=
\lambda\,\mathcal{F}(w;\lambda).
\]

Thus the generating function satisfies the simple differential equation
\begin{equation}\label{eq:gen-ode}
\partial_w \mathcal{F}(w;\lambda)
=
\lambda\,\mathcal{F}(w;\lambda),
\end{equation}
to be understood formally in $\lambda$.

\subsection{Solution and structural interpretation}

Formally solving~\eqref{eq:gen-ode} gives
\begin{equation}\label{eq:gen-solution}
\mathcal{F}(w;\lambda)
=
e^{\lambda w}\,\mathcal{F}(0;\lambda).
\end{equation}

Since $F_n(0)=0$ for $n\ge2$ by construction,
we have
\[
\mathcal{F}(0;\lambda)=F_1(0).
\]

Hence the entire hierarchy is generated by exponentiating
the backbone variable $w$ against the deformation parameter $\lambda$,
with the analytic content encoded solely in the seed.

\subsection{CL/SL decomposition at the generating level}

Since each $F_n$ admits the decomposition
\[
F_n(w)
=
\frac12 A(n;w)
-\frac{i}{2} B(n;w),
\]
the generating function likewise decomposes into real and imaginary parts:
\begin{equation}\label{eq:gen-CLSL}
\mathcal{F}(w;\lambda)
=
\frac12 \mathcal{A}(w;\lambda)
-\frac{i}{2} \mathcal{B}(w;\lambda),
\end{equation}
where
\[
\mathcal{A}(w;\lambda)
:=
\sum_{n=1}^{\infty} A(n;w)\lambda^{n-1},
\qquad
\mathcal{B}(w;\lambda)
:=
\sum_{n=1}^{\infty} B(n;w)\lambda^{n-1}.
\]

Taking real and imaginary parts of the generating equation
\(
\partial_w \mathcal{F}=\lambda \mathcal{F}
\)
gives
\begin{equation}\label{eq:gen-CLSL-ode}
\partial_w \mathcal{A}
=
\lambda\,\mathcal{A},
\qquad
\partial_w \mathcal{B}
=
\lambda\,\mathcal{B}.
\end{equation}

\medskip

Thus the CL/SL dichotomy reflects not a structural difference
in the hierarchy itself, but merely the real and imaginary
projections of a single generating deformation.

\begin{remark}[Backbone and generating viewpoints]
The generating formulation does not introduce a new hierarchy,
but merely reorganizes the recursion backbone.
Indeed, the relation $\partial_w F_{n+1}=F_n$ already determines
the tower once the seed $F_1$ is fixed.
The generating series $\mathcal{F}(w;\lambda)$ simply packages this
structure into a single deformation object,
whose real and imaginary projections reproduce the CL/SL families.
\end{remark}

\medskip

This generating viewpoint will provide a convenient starting point
for examining broader analytic structures associated with the hierarchy.

\section{Outlook and Analytical Remarks}

\subsection{Toward a Unified Generating Object}

The generating viewpoint developed in Section~4 suggests
the existence of a more fundamental object from which both
the circular Clausen structures and their elliptic
counterparts arise naturally.
Possible candidates include logarithmic derivatives of
Jacobi theta functions, Weierstrass sigma-type generating
functions, or Lerch-type analytic generators.
A systematic formulation of such a unified generating
object, capable of producing both CL- and SL-type
hierarchies, remains an interesting direction for future
investigation.

\subsection{Analytical Remarks}

The elliptic SL-type components involve phase functions
whose branches must be chosen consistently within the
fundamental domain.
Although the constructions presented here remain well
defined under the standard analytic continuation of
Jacobi theta functions, a detailed study of branch
structures, convergence properties, and analytic domains
will require a more systematic treatment.

\subsection{Future Directions}

Several natural extensions suggest themselves.
These include a fuller development of the elliptic
SL-type hierarchy, relations with Weierstrass-type
formulations, and the possible construction of a unified
generating framework connecting polylogarithmic,
Clausen, and elliptic structures.

The relation to Weierstrass-type expansions and the systematic appearance of Hurwitz numbers in the Laurent coefficients will be investigated in future work.

We hope that the present recursion-based perspective
provides a useful starting point for such developments.

\begin{remark}
The entire hierarchy may be viewed as a universal recursion generated by a choice of seed function $S_{\star}$, interpolating between polylogarithmic, trigonometric, and elliptic regimes. 

\[
F^{(\star)}_1(z)
=
\log S_{\star}(z),
\qquad
\partial_z F^{(\star)}_{n+1}(z)=F^{(\star)}_n(z),
\]
where the seed function $S_{\star}$ is chosen as
\[
S_\star(z)
=
\begin{cases}
1-e^{iz} & (\text{polylog seed})\\[4pt]
2\sin(z/2) & (\text{circular seed})\\[4pt]
\widetilde{\vartheta}_1(z|\tau) & (\text{elliptic seed})
\end{cases}
\]
\end{remark}


\section*{Acknowledgments}
The author thanks the collaborative discussions carried out within the \emph{hoge \& fuga} series (2026).

\bibliographystyle{plain}

\end{document}